\documentclass[11pt]{amsart}
\usepackage{amsthm,amsmath,amssymb}
\usepackage{stmaryrd,mathrsfs}
\usepackage{enumerate}
\usepackage[T1]{fontenc}
\usepackage{esint}
\usepackage{tikz}

\allowdisplaybreaks

\newcommand{\ud}[0]{\,\mathrm{d}}

\newcommand{\abs}[1]{|#1|}

\newcommand{\Norm}[2]{\|#1\|_{#2}}

\newcommand{\ave}[1]{\langle #1\rangle}

\newcommand{\BMO}[0]{\operatorname{BMO}}
\newcommand{\CMO}[0]{\operatorname{CMO}}

\newcommand{\supp}[0]{\operatorname{supp}}
\newcommand{\loc}[0]{\operatorname{loc}}

\newcommand{\testi}{{\mathcal S}}

\newcommand{\R}{\mathbb{R}}
\newcommand{\C}{\mathbb{C}}
\newcommand{\N}{\mathbb{N}}

\newcommand{\eps}[0]{\varepsilon}

\swapnumbers \numberwithin{equation}{section}

\theoremstyle{plain}
\newtheorem{theorem}[equation]{Theorem}
\newtheorem{proposition}[equation]{Proposition}
\newtheorem{corollary}[equation]{Corollary}
\newtheorem{lemma}[equation]{Lemma}
\newtheorem{conjecture}[equation]{Conjecture}

\theoremstyle{definition}
\newtheorem{definition}[equation]{Definition}

\theoremstyle{remark}
\newtheorem{remark}[equation]{Remark}

\makeatletter
\@namedef{subjclassname@2020}{%
  \textup{2020} Mathematics Subject Classification}
 \def\@textbottom{\vskip \z@ \@plus 1pt}
 \let\@texttop\relax
\makeatother

\begin{document}

\title[Extrapolation of compactness]{Extrapolation of compactness on\\ weighted spaces II: Off-diagonal and limited range estimates}

\author[T. \ Hyt\"onen and S. Lappas]{Tuomas Hyt\"onen and Stefanos Lappas}
\address{Department of Mathematics and Statistics, P.O.Box~68 (Pietari Kalmin katu~5), FI-00014 University of Helsinki, Finland}
\email{tuomas.hytonen@helsinki.fi}
\email{stefanos.lappas@helsinki.fi}





\maketitle


\begin{abstract}
In a previous paper by one of us, a ``compact version'' of Rubio de Francia's weighted extrapolation theorem was proved, which allows one to extrapolate the compactness of an operator from just one space to the full range of weighted spaces, where this operator is bounded. In this paper, we obtain generalizations of this extrapolation of compactness for operators that are bounded from one space to a different one (``off-diagonal estimates'') or only in a limited range of the $L^p$ scale. As applications, we easily recover recent results on the weighted compactness of commutators of fractional integrals and pseudo-differential operators, and obtain new results about the weighted compactness of Bochner--Riesz multipliers.
\end{abstract}

\section{Introduction}

By a {\em weight} we mean a locally integrable function $w\in L^1_{\loc}(\R^d)$ that is positive almost everywhere. We recall the definitions of $A_p(\R^d)$,
$A_{p,q}(\R^d)$, and $RH_r(\R^d)$ classes of weights first introduced by Muckenhoupt \cite{Muckenhoupt:Ap}, Muckenho\-upt--Wheeden \cite{MW:74}, and Gehring \cite{Gehring}:

\begin{definition}\label{def:Muchenhoupt weights}
A weight $w\in L_{\loc}^1(\R^d)$ is called a Muckenhoupt $A_p(\R^d)$ weight (or $w\in A_p(\R^d)$) if 
\begin{equation*}
\begin{split}
  &[w]_{A_p}:=\sup_Q\ave{w}_Q\ave{w^{-\frac{1}{p-1}}}_Q^{p-1}<\infty,\qquad 1<p<\infty, \\
  &[w]_{A_1}:=\sup_Q\ave{w}_Q\Norm{w^{-1}}{L^\infty(Q)}<\infty,\qquad p=1,
\end{split}
\end{equation*}
where the supremum is taken over all cubes $Q\subset\R^d$, and $\ave{w}_Q:=\abs{Q}^{-1}\int_Q w$. A weight $w$ is called an $A_{p,q}(\R^d)$ weight (or $w\in A_{p,q}(\R^d)$) if
\begin{equation*}
  [w]_{A_{p,q}}:=\sup_Q\ave{w^q}_Q^{1/q}\ave{w^{-p'}}_Q^{1/p'}<\infty,\qquad 1<p\leq q<\infty,
\end{equation*}
where $p':=p/(p-1)$ denotes the conjugate exponent.

We say that $w$ belongs to the reverse H\"older class $RH_r(\R^d)$ (or $w\in RH_r(\R^d)$) if
\begin{equation*}
  [w]_{RH_r}:=\sup_Q\ave{w^r}_Q^{1/r}\ave{w}_Q^{-1}<\infty,\qquad 1<r<\infty.
\end{equation*}
\end{definition}

As we will work in the weighted setting, we consider weighted Lebesgue spaces
\begin{equation*}
  L^p(w):=\Big\{f:\R^d\to\C\text{ measurable }\Big|\ \Norm{f}{L^p(w)}:=\Big(\int_{\R^d}\abs{f}^p w\Big)^{1/p}<\infty\Big\}.
\end{equation*}

The classes $A_p(\R^d)$ and $A_{p,q}(\R^d)$ were introduced to study the weighted norm inequalities for the {\em Hardy--Littlewood maximal function} and for {\em fractional integral operators}, respectively; see \cite{Muckenhoupt:Ap,MW:74}. On the other hand, the reverse H\"older classes $RH_r(\R^d)$ were introduced to study the {\em $L^p$-integrability} of the partial derivatives of a quasiconformal mapping; see \cite{Gehring}. The close connection between these weight classes is well-known since the work \cite{CF}.

The following theorem of Rubio de Francia \cite{Rubio:factorAp} on the extrapolation of {\em boundedness} on weighted spaces is one of the highlights in the theory of weighted norm inequalities:

\begin{theorem}[\cite{Rubio:factorAp}]\label{thm:RdF}
Let $1\leq \lambda<p_1<\infty$, and $T$ be a linear operator simultaneously defined and bounded on $L^{p_1}(\tilde w)$ for {\bf all} $\tilde w\in A_{p_1/\lambda}(\R^d)$, with the operator norm dominated by some increasing function of $[\tilde w]_{A_{p_1/\lambda}}$.
Then $T$ is also defined and bounded on $L^p(w)$ for all $p\in(\lambda,\infty)$ and all $w\in A_{p/\lambda}(\R^d)$.
\end{theorem}

In a recent paper, one of the authors \cite{Hyt} provided the following version for extrapolation of {\em compactness}:

\begin{theorem}\label{thm:RdFcompact}
In the setting of Theorem \ref{thm:RdF}, suppose in addition that $T$ is compact on $L^{p_1}(w_1)$ for {\bf some} $w_1\in A_{p_1/\lambda}(\R^d)$. Then $T$ is compact on $L^p(w)$ for all $p\in(1,\infty)$ and all $w\in A_{p/\lambda}(\R^d)$.
\end{theorem}

In this paper, we seek to prove extrapolation of compactness theorems for operators that are bounded either from $L^p$ to $L^q$, for possibly different exponents $1<p\leq q<\infty$ or on $L^p$, for a limited range of the exponent $p$. For these type of operators the following versions of Rubio de Francia's extrapolation theorems are available: 

\begin{theorem}[\cite{HMS}, Harboure--Mac\'ias--Segovia]\label{thm:Off-dig.extrp.}
Let $T$ be a linear operator defined and bounded from $L^{p_1}(\tilde w^{p_1})$ to $L^{q_1}(\tilde w^{q_1})$ for {\bf some} $1<p_1\leq q_1<\infty$ and {\bf all} $\tilde w\in A_{p_1,q_1}(\R^d)$.
Then $T$ is also defined and bounded from $L^p(w^p)$ to $L^q(w^q)$ for all $1<p\leq q<\infty$ such that $\frac{1}{p}-\frac{1}{q}=\frac{1}{p_1}-\frac{1}{q_1}$ and all $w\in A_{p,q}(\R^d)$.
\end{theorem}

This applies to the study of the fractional integral operators, also known as the {\em Riesz potentials} (see Section \ref{comm. fr. int. op.}). A version of Theorem \ref{thm:Off-dig.extrp.}, with sharp constants is due to Lacey--Moen--Per\'ez--Torres \cite{Lacey2010}. A more general version, with sharp constants and including values of $0<q<p$, was given by Duoandikoetxea \cite{D2011}.

\begin{theorem}[\cite{AM}, Theorem 4.9 of Auscher--Martell]\label{thm:limited range extrp.}
Let $1\leq p_{-}<p_{+}<\infty$, and $T$ be a linear operator simultaneously defined and bounded on $L^{p_1}(\tilde w)$ for {\bf some} $1\leq p_{-}\leq p_1\leq p_{+}<\infty$ and {\bf all} $\tilde w\in A_{p_1/p_{-}}(\R^d)\cap RH_{(p_{+}/p_1)'}(\R^d)$. Then $T$ is also defined and bounded on $L^p(w)$ for all $p\in(p_{-},p_{+})$ and all $w\in A_{p/p_{-}}(\R^d)\cap RH_{(p_{+}/p)'}(\R^d)$.
\end{theorem}

See also \cite{CUMP:book} where these extrapolation theorems and some others are discussed. In \cite{CUMP:book}, Theorems \ref{thm:Off-dig.extrp.} and \ref{thm:limited range extrp.} are stated in terms of non-negative, measurable pairs of functions $(f,g)$. The reason is that one does not need to work with specific operators since nothing about the operators themselves is used (like linearity or sublinearity) and they play no role. However, we work with linear operators since an abstract compactness result that we will use in order to prove Theorems \ref{thm:Off-dig.extrp.compact} and \ref{thm:limited range extrp.compact} below holds for linear operators (see Theorem \ref{thm:CwKa} of Cwikel--Kalton).

In this paper, we extend the results of \cite{Hyt} about the extrapolation of compactness to the setting of Theorems \ref{thm:Off-dig.extrp.} and \ref{thm:limited range extrp.}:

\begin{theorem}\label{thm:Off-dig.extrp.compact}
In the setting of Theorem \ref{thm:Off-dig.extrp.}, suppose in addition that $T$ is compact from $L^{p_1}(w_1^{p_1})$ to $L^{q_1}(w_1^{q_1})$ for {\bf some} $w_1\in A_{p_1,q_1}(\R^d)$. Then $T$ is compact from $L^p(w^p)$ to $L^q(w^q)$ for all $1<p\leq q<\infty$ such that
$\frac{1}{p}-\frac{1}{q}=\frac{1}{p_1}-\frac{1}{q_1}$ and all $w\in A_{p,q}(\R^d)$.
\end{theorem}

\begin{theorem}\label{thm:limited range extrp.compact}
In the setting of Theorem \ref{thm:limited range extrp.}, suppose in addition that $T$ is compact on $L^{p_1}(w_1)$ for {\bf some} $w_1\in A_{p_1/p_{-}}(\R^d)\cap RH_{(p_{+}/p_1)'}(\R^d)$. Then $T$ is compact on $L^{p}(w)$ for all $p\in(p_{-},p_{+})$ and all $w\in A_{p/p_{-}}(\R^d)\cap RH_{(p_{+}/p)'}(\R^d)$.
\end{theorem}

\begin{remark}
Theorems \ref{thm:limited range extrp.} and \ref{thm:limited range extrp.compact} remain true if $p_{+}=\infty$. In this case the reverse H\"older condition on $w$ is vacuous.
\end{remark}

When $w_1\equiv1$, Theorems \ref{thm:Off-dig.extrp.compact} and \ref{thm:limited range extrp.compact} say that we can obtain weighted compactness if the weighted boundedness and unweighted compactness are already known. This case is relevant to all our applications in Sections \ref{comm. fr. int. op.} and \ref{comm. br. mult.}.

The paper is organized as follows: in Section \ref{sec:main results} we present the proofs of Theorems \ref{thm:Off-dig.extrp.compact} and \ref{thm:limited range extrp.compact} by collecting some previously known results and taking some auxiliary results for granted. Section \ref{sec:prop} is dedicated to the proofs of these auxiliary results (see Propositions \ref{prop:LpvInterm} and \ref{prop:limitedrange}). In Sections \ref{comm. fr. int. op.} and \ref{comm. br. mult.} we provide several applications of our main results. In particular, we obtain previously known results for the commutators of fractional integral operators and a new result for the commutators of {\em Bochner--Riesz multipliers}. In Section \ref{sec:Ap(varphi) weights} we develop and apply yet another variant for extrapolation of compactness for a special class of weights related to the commutators of {\em pseudo-differential operators} with smooth symbols.

\subsection*{Notation} Throughout the paper, we denote by $C$ a positive constant which is independent of the main parameters but it may change at each occurrence, and we write $f\lesssim g$ if $f\leq Cg$. The term cube will always refer to a cube $Q\subset\R^d$ and $|Q|$ will denote its Lebesgue measure. Recall from Definition \ref{def:Muchenhoupt weights} that $\ave{w}_Q$ denotes $\abs{Q}^{-1}\int_Q w$, the average of $w$ over $Q$, and $p'$ is the conjugate exponent to $p$, that is $p':=p/(p-1)$.

\subsection*{Acknowledgements} Both authors were supported by the Academy of Finland through the grant No. 314829. The second author wishes to thank his doctoral supervisor Prof. Tuomas Hyt\"onen for helpful discussions. Also, the second author would like to thank the Foundation for Education and European Culture (Founders Nicos and Lydia Tricha) for their financial support during the academic years 2017--2020.

\section{Auxiliary results and the proofs of Theorems \ref{thm:Off-dig.extrp.compact} and \ref{thm:limited range extrp.compact} }\label{sec:main results}

We collect the results from which the proofs of Theorems \ref{thm:Off-dig.extrp.compact} and \ref{thm:limited range extrp.compact} follow. Our main abstract tool is the following theorem of Cwikel--Kalton \cite{CwKa}:

\begin{theorem}[\cite{CwKa}]\label{thm:CwKa}
Let $(X_0,X_1)$ and $(Y_0,Y_1)$ be Banach couples and $T$ be a linear operator such that
$T:X_0+X_1\to Y_0+Y_1$ and $T:X_j\to Y_j$ boundedly for $j=0,1$.
Suppose moreover that $T:X_1\to Y_1$ is compact.
Let $[\ ,\ ]_\theta$ be the complex interpolation functor of Calder\'on.
Then also $T:[X_0,X_1]_\theta\to[Y_0,Y_1]_\theta$ is compact for $\theta\in(0,1)$ under {\bf any} of the following four side conditions:
\
\begin{enumerate}
  \item\label{it:UMD} $X_1$ has the UMD (unconditional martingale differences) property,
  \item\label{it:Xinterm} $X_1$ is reflexive, and $X_1=[X_0,E]_\alpha$ for some Banach space $E$ and $\alpha\in(0,1)$,
  \item\label{it:Yinterm} $Y_1=[Y_0,F]_\beta$ for some Banach space $F$ and $\beta\in(0,1)$,
  \item\label{it:lattice} $X_0$ and $X_1$ are both complexified Banach lattices of measurable functions on a common measure space.
\end{enumerate}
\end{theorem}

(We have swapped the roles of the indices $0$ and $1$ in comparison to \cite{CwKa}. For the UMD property, see \cite[Ch. 4]{HNVW1}.)
We will use Theorem \ref{thm:CwKa} in the following special settings: 

\begin{proposition}\label{prop:LpvInterm}
Suppose that $1<p\leq q<\infty$, $1<p_1\leq q_1<\infty$ and $v\in A_{p,q}(\R^d)$, $v_1\in A_{p_1,q_1}(\R^d)$. Then
\begin{equation*}
  [L^{p_0}({v_0}^{p_0}),L^{p_1}({v_1}^{p_1})]_\gamma=L^p(v^p)\,\,\,\text{and}\,\,\,[L^{q_0}({v_0}^{q_0}),L^{q_1}({v_1}^{q_1})]_\gamma=L^q(v^q)
\end{equation*}
for some $1<p_0\leq q_0<\infty$, $v_0\in A_{p_0,q_0}(\R^d)$, and $\gamma\in(0,1)$. Moreover, if $\frac{1}{p}-\frac{1}{q}=\frac{1}{p_1}-\frac{1}{q_1}$ we can choose $p_0,q_0$ in such a way that $\frac{1}{p}-\frac{1}{q}=\frac{1}{p_0}-\frac{1}{q_0}$.
\end{proposition}

\begin{proposition}\label{prop:limitedrange}
Suppose that $1\leq p_{-}<p_{+}<\infty$, $q_1\in[p_{-},p_{+}]$, $q\in(p_{-},p_{+})$ and
\begin{equation*}
v\in A_{q/p_{-}}(\R^d)\cap RH_{(p_{+}/q)'}(\R^d),\qquad v_1\in A_{q_1/p_{-}}(\R^d)\cap RH_{(p_{+}/q_1)'}(\R^d). 
\end{equation*}
Then
\begin{equation*}
  [L^{q_0}(v_0),L^{q_1}(v_1)]_{\gamma}=L^q(v)
\end{equation*}
for some $q_0\in(p_{-},p_{+})$, $v_0\in A_{q_0/p_{-}}(\R^d)\cap RH_{(p_{+}/q_0)'}(\R^d)$, and $\gamma\in(0,1)$.
\end{proposition}

We postpone the proofs of Propositions \ref{prop:LpvInterm} and \ref{prop:limitedrange} to the following section. The verifications of these propositions are the only components of the proofs of Theorems  \ref{thm:Off-dig.extrp.compact} and \ref{thm:limited range extrp.compact} that require actual computations, rather than just a soft application of known results. 

\begin{lemma}\label{lem:lemOk}
If $p_j\in[1,\infty)$ and $w_j$ are weights, then the spaces $X_j=Y_j=L^{p_j}(w_j)$ satisfy the condition \eqref{it:lattice} of Theorem \ref{thm:CwKa}.
\end{lemma}

\begin{proof}
\eqref{it:lattice}: It is easy to see that $X_j=Y_j=L^{p_j}(w_j)$ are complexified Banach lattices of measurable functions on the common measure space $\R^d$.
\end{proof}

\begin{remark}
If $p_j\in(1,\infty)$ then the conditions \eqref{it:UMD}, \eqref{it:Xinterm} and \eqref{it:Yinterm} of Theorem \ref{thm:CwKa} are also satisfied by the spaces $X_j=Y_j=L^{p_j}(w_j)$ (see \cite{Hyt}). For applications of Theorem \ref{thm:CwKa} to these concrete spaces, this is of course more than sufficient. We would only need one of the four side conditions, but in fact we have them all.
\end{remark}

We can now give the proof of our main results:

\begin{proof}[Proof of Theorem \ref{thm:Off-dig.extrp.compact}]
Recall that the assumptions, and hence the conclusions, of Theorem \ref{thm:Off-dig.extrp.} are in force. In particular, $T:L^p(w^p)\to L^q(w^q)$ is a bounded linear operator for all $1<p\leq q<\infty$ such that $\frac{1}{p}-\frac{1}{q}=\frac{1}{p_1}-\frac{1}{q_1}$ and all $w\in A_{p,q}(\R^d)$. In addition, it is assumed that $T:L^{p_1}(w_1^{p_1})\to L^{q_1}(w_1^{q_1})$ is a compact operator for some $1<p_1\leq q_1<\infty$ and some $w_1\in A_{p_1,q_1}(\R^d)$. We need to prove that $T:L^p(w^p)\to L^q(w^q)$ is actually compact for all $1<p\leq q<\infty$ such that $\frac{1}{p}-\frac{1}{q}=\frac{1}{p_1}-\frac{1}{q_1}$ and all $w\in A_{p,q}(\R^d)$. Now, fix some $1<p\leq q<\infty$ and $w\in A_{p,q}(\R^d)$. By Proposition \ref{prop:LpvInterm}, we have
\begin{equation*}
  L^p(w^p)=[L^{p_0}({w_0}^{p_0}),L^{p_1}({w_1}^{p_1})]_\theta\,\,\,\text{and}\,\,\,L^q(w^q)=[L^{q_0}({w_0}^{q_0}),L^{q_1}({w_1}^{q_1})]_\theta
\end{equation*}
for some $1<p_0\leq q_0<\infty$, some $w_0\in A_{p_0,q_0}(\R^d)$, some $\theta\in(0,1)$ and $\frac{1}{p}-\frac{1}{q}=\frac{1}{p_0}-\frac{1}{q_0}$. Writing $X_j=L^{p_j}(w_j^{p_j})$ and $Y_j=L^{q_j}(w_j^{q_j})$, we know that $T:X_0+X_1\to Y_0+Y_1$ and $T:X_j\to Y_j$ is bounded ($T:L^{\tilde p}(w^{\tilde p})\to L^{\tilde q}(w^{\tilde q})$ is a bounded linear operator for all $1<\tilde p\leq\tilde q<\infty$ such that $\frac{1}{\tilde p}-\frac{1}{\tilde q}=\frac{1}{\tilde p_1}-\frac{1}{\tilde q_1}$ and all $w\in A_{\tilde p,\tilde q}(\R^d)$ by Theorem \ref{thm:Off-dig.extrp.}), and that $T:X_1\to Y_1$ is compact (since this was assumed). By Lemma \ref{lem:lemOk}, the last condition \eqref{it:lattice} of Theorem \ref{thm:CwKa} is also satisfied by these spaces $X_j=L^{p_j}(w_j^{p_j})$ and $Y_j=L^{q_j}(w_j^{q_j})$.  By Theorem \ref{thm:CwKa}, it follows that $T:L^p(w^p)=[X_0,X_1]_\theta\to L^q(w^q)=[Y_0,Y_1]_\theta$ is also compact.
\end{proof}

\begin{proof}[Proof of Theorem \ref{thm:limited range extrp.compact}]
Recall that the assumptions, and hence the conclusions, of Theorem \ref{thm:limited range extrp.} are in force. In particular, $T$ is a bounded linear operator on $L^p(w)$ for all $p\in(p_{-},p_{+})$ and all $w\in A_{p/p_{-}}(\R^d)\cap RH_{(p_{+}/p)'}(\R^d)$. In addition, it is assumed that $T$ is a compact operator on $L^{p_1}(w_1)$ for some $p_1\in[p_{-},p_{+}]$ and some $w_1\in A_{p_1/p_{-}}(\R^d)\cap RH_{(p_{+}/p_1)'}(\R^d)$. We need to prove that $T$ is actually compact on $L^p(w)$ for all $p\in(p_{-},p_{+})$ and all $w\in A_{p/p_{-}}(\R^d)\cap RH_{(p_{+}/p)'}(\R^d)$. Now, fix some $p\in(p_{-},p_{+})$ and $w\in A_{p/p_{-}}(\R^d)\cap RH_{(p_{+}/p)'}(\R^d)$. By Proposition \ref{prop:limitedrange}, we have
\begin{equation*}
  L^p(w)=[L^{p_0}(w_0),L^{p_1}(w_1)]_\theta
\end{equation*}
for some $p_0\in(p_{-},p_{+})$, some $w_0\in A_{p_0/p_{-}}(\R^d)\cap RH_{(p_{+}/p_0)'}(\R^d)$ and some $\theta\in(0,1)$. Writing $X_j=Y_j=L^{p_j}(w_j)$, we know that $T:X_0+X_1\to Y_0+Y_1$, that $T:X_j\to Y_j$ is bounded (since $T$ is bounded on all $L^q(w)$ with $q\in(p_{-},p_{+})\cup\{p_1\}$ and $w\in A_{q/p_{-}}(\R^d)\cap RH_{(p_{+}/q)'}(\R^d)$ by the assumptions and the conclusion of Theorem \ref{thm:limited range extrp.}), and that $T:X_1\to Y_1$ is compact (since this was assumed). By Lemma \ref{lem:lemOk}, the last condition \eqref{it:lattice} of Theorem \ref{thm:CwKa} is also satisfied by these spaces $X_j=Y_j=L^{p_j}(w_j)$. By Theorem \ref{thm:CwKa}, it follows that $T$ is also compact on $[X_0,X_1]_\theta=[Y_0,Y_1]_\theta=L^p(w)$.
\end{proof}

\section{The Proofs of Propositions \ref{prop:LpvInterm} and \ref{prop:limitedrange}}\label{sec:prop}

To complete the proofs of Theorems \ref{thm:Off-dig.extrp.compact} and \ref{thm:limited range extrp.compact}, it remains to verify Propositions \ref{prop:LpvInterm} and \ref{prop:limitedrange}. We quote two more classical results:

\begin{proposition}[\cite{RDF1985, Gehring, JN1991}]\label{prop:weights} The following statements hold:
\
\begin{enumerate}
\item \label{weights prop. 1} \ (\cite[Theorem 1.14]{RDF1985}) If $1<p<\infty$, we have $w\in A_p(\R^d)\Longleftrightarrow w^{1-p'}\in
A_{p'}(\R^d)$.
\item \label{weights prop. 2} \ (\cite[Theorem 2.6]{RDF1985}) If $w\in A_p(\R^d)$, $1<p<\infty$, then there exists $1<q<p$ such
that $w\in A_q(\R^d)$.
\item \label{weights prop. 3} \ (\cite[Lemma 3]{Gehring}) If $w\in RH_q(\R^d)$, $1<q<\infty$, then there exists $q<p<\infty$ such
that $w\in RH_p(\R^d)$.
\item \label{weights prop. 4} \ If $w\in A_{p,q}(\R^d)$, $1<p\leq q<\infty$, then $w^q\in A_{1+q/p'}(\R^d)$ and $w^{-p'}\in A_{1+p'/q}(\R^d)$, where $\frac{1}{p}+\frac{1}{p'}=1$.
\item \label{weights prop. 5} \ (\cite[Statement (P6)]{JN1991}) If $1<q,s<\infty$, then $w\in A_q(\R^d) \cap RH_s(\R^d)\Longleftrightarrow w^{s}\in A_{s\,(q-1)+1}(\R^d)$.
\end{enumerate}
\end{proposition}

\begin{proof}
We only prove property \eqref{weights prop. 4}. Notice that $w\in A_{p,q}(\R^d)\Longleftrightarrow w^q\in A_r(\R^d)$, with $[w]_{A_{p,q}}=[w^q]_{A_r}$, where 
\begin{equation*}
  r:=1+q/p'.
\end{equation*}
The proof of $w^{-p'}\in A_{1+p'/q}(\R^d)$ follows in a similar fashion.
\end{proof}

\begin{theorem}[\cite{BL}, Theorem 5.5.3]\label{thm:SW}
If $q_0,q_1\in[1,\infty)$ and $w_0,w_1$ are two weights, then for all $\theta\in(0,1)$ we have
\begin{equation*}
  [L^{q_0}(w_0),L^{q_1}(w_1)]_\theta=L^q(w),
\end{equation*}
where
\begin{equation}\label{eq:convexity}
  \frac{1}{q}=\frac{1-\theta}{q_0}+\frac{\theta}{q_1},\qquad
  w^{\frac{1}{q}}=w_0^{\frac{1-\theta}{q_0}}w_1^{\frac{\theta}{q_1}}.
\end{equation}
\end{theorem}

In order to connect Theorem \ref{thm:SW} with the $A_{p,q}(\R^d)$ and $A_{q/p_{-}}(\R^d)\cap RH_{(p_{+}/q)'}(\R^d)$ weights, we need:

\begin{lemma}\label{lem:main1}
Let $1<p_1\leq q_1<\infty$, $1<p\leq q<\infty$, $w_1\in A_{p_1,q_1}(\R^d)$, $w\in A_{p,q}(\R^d)$. Then there exist $1<p_0\leq q_0<\infty$, $w_0\in A_{p_0,q_0}(\R^d)$, and $\theta\in(0,1)$ such that the conclusion of Theorem \ref{thm:SW} holds, i.e.,
\begin{equation*}
  [L^{p_0}({w_0}^{p_0}),L^{p_1}({w_1}^{p_1})]_\theta=L^p(w^p),\qquad[L^{q_0}({w_0}^{q_0}),L^{q_1}({w_1}^{q_1})]_\theta=L^q(w^q)
\end{equation*}
where
\begin{equation*}
  \frac{1}{p}=\frac{1-\theta}{p_0}+\frac{\theta}{p_1},\qquad\frac{1}{q}=\frac{1-\theta}{q_0}+\frac{\theta}{q_1},\qquad
  w=w_0^{1-\theta}w_1^{\theta}.
\end{equation*}
\end{lemma}

\begin{proof}
Note that the choice of $\theta\in(0,1)$ determines 
\begin{equation*}
  p_0=p_0(\theta)=\frac{1-\theta}{\frac{1}{p}-\frac{\theta}{p_1}},\quad
  q_0=q_0(\theta)=\frac{1-\theta}{\frac{1}{q}-\frac{\theta}{q_1}},\quad
  w_0=w_0(\theta)=w^{\frac{1}{1-\theta}}w_1^{-\frac{\theta}{1-\theta}},
\end{equation*}
so it remains to check that we can choose $\theta\in(0,1)$ so that $1<p_0\leq q_0<\infty$ and $w_0\in A_{p_0,q_0}(\R^d)$. Since $1<p_0(0)=p\leq q=q_0(0)<\infty$, the first condition is obvious for small enough $\theta>0$ by continuity. 

We need to check that $w_0\in A_{p_0,q_0}(\R^d)$, so we consider a cube $Q$ and write
\begin{equation*}
  \ave{w_0^{q_0}}_Q^{\frac{1}{q_0}}\ave{w_0^{-p_0'}}_Q^{\frac{1}{p_0'}}
  =\ave{w^{\frac{q_0}{1-\theta}}w_1^{-\frac{q_0\cdot\theta}{1-\theta}}}_Q^{\frac{1}{q_0}}
  \ave{w^{-\frac{p_0'}{1-\theta}}w_1^{\frac{p_0'\cdot\theta}{1-\theta}}}_Q^{\frac{1}{p_0'}},
\end{equation*}
where $p_0':=p_0/(p_0-1)$ denotes the conjugate exponent of $p_0$.

In the first average, we use H\"older's inequality with exponents $1+\eps^{\pm 1}$, and in the second with exponents $1+\delta^{\pm 1}$ to get
\begin{equation*}
\begin{split}
  &\leq \ave{w^{\frac{q_0(1+\eps)}{1-\theta}}}_Q^{\frac{1}{q_0(1+\eps)}} \ave{w_1^{-\frac{q_0\theta(1+\eps)}{\eps(1-\theta)}}}_Q^{\frac{\eps}{q_0(1+\eps)}}  \\ &\qquad\times\ave{w^{-\frac{p_0'(1+\delta)}{1-\theta}}}_Q^{\frac{1}{p_0'(1+\delta)}} \ave{w_1^{\frac{p_0'\theta(1+\delta)}{\delta(1-\theta)}}}_Q^{\frac{\delta}{p_0'(1+\delta)}}.
\end{split}
\end{equation*}
If we choose $\eps=\frac{\theta q}{p_1'}$ and $\delta=\frac{\theta p'}{q_1}$, the previous line takes the form
\begin{equation}\label{eq:beforeRHI1}
\begin{split}
  &= \bigg(\ave{w^{\frac{q_0(p_1'+\theta q)}{p_1'(1-\theta)}}}_Q^{p_1'} \ave{w_1^{-p_1's(\theta)}}_Q^{\theta q}\bigg)^\frac{1}{q_0(p_1'+\theta q)}  \\
  &\qquad\times\bigg(\ave{w^{-\frac{p_0'(q_1+\theta p')}{q_1(1-\theta)}}}_Q^{q_1} \ave{ w_1^{q_1 u(\theta)}}_Q^{\theta p'}\bigg)^\frac{1}{p_0'(q_1+\theta p')}  \\
  &=\ave{(w^{q})^{r(\theta)}}_Q^{\frac{1}{r(\theta)q(1-\theta)}} \ave{(w_1^{-p_1'})^{s(\theta)}}_Q^{\frac{\theta q}{q_0(p_1'+\theta q)}}  \\
  &\qquad\times\ave{(w^{-p'})^{t(\theta)}}_Q^{\frac{1}{t(\theta)p'(1-\theta)}} \ave{ (w_1^{q_1})^{u(\theta)}}_Q^{\frac{\theta p'}{p_0'(q_1+\theta p')}}
\end{split}
\end{equation}
where
\begin{equation*}
  r(\theta):=\frac{q_1(p_1'+\theta q)}{p_1'(q_1-\theta q)},\qquad
  s(\theta):=\frac{q_0(\theta)(p_1'+\theta q)}{q p_1'(1-\theta)}
\end{equation*}
and
\begin{equation*}
  t(\theta):=\frac{p_1'(q_1+\theta p')}{q_1(p_1'-\theta p')},\qquad
  u(\theta):=\frac{p_0(\theta)'(q_1+\theta p')}{ p' q_1(1-\theta)}.
\end{equation*}

The strategy to proceed is to use the reverse H\"older inequality for $A_v(\R^d)$ weights due to Coifman--Fefferman \cite{CF}, which says that each $W\in A_v(\R^d)$ satisfies
\begin{equation}\label{eq:RHI}
  \ave{W^t}_Q^{1/t}\lesssim \ave{W}_Q
\end{equation}
for all $t\leq 1+\eta$ and for some $\eta>0$ depending only on $[W]_{A_v}$. (For a sharp quantitative version, see \cite[Theorem 2.3]{HytPer}.)

Recalling that $p_0(0)=p$ and $q_0(0)=q$, we see that $r(0)=s(0)=t(0)=u(0)=1$. By continuity, given any $\eta>0$, we find that 
\begin{equation*}
\max(r(\theta),s(\theta),t(\theta),u(\theta))\leq 1+\eta\,\,\,\text{for all small enough}\,\,\,\theta>0.
\end{equation*}
By property \eqref{weights prop. 4} of Proposition \ref{prop:weights} each of the four functions $w^q\in A_{1+\frac{q}{p'}}(\R^d)$, $w_1^{-p_1'}\in A_{1+p_1'/q_1}(\R^d)$, $w^{-p'}\in A_{1+p'/q}(\R^d)$ and $w_1^{q_1}\in A_{1+q_1/p_1'}(\R^d)$ satisfies the reverse H\"older inequality (\ref{eq:RHI}) for all $t\leq 1+\eta$ and for some $\eta>0$. Thus, for all small enough $\theta>0$, we have
\begin{equation*}
\begin{split}
  \eqref{eq:beforeRHI1}
  &\lesssim  \ave{w^{q}}_Q^{r(\theta)\frac{1}{r(\theta)q(1-\theta)}} \ave{w_1^{-p_1'}}_Q^{s(\theta)\frac{\theta q}{q_0(p_1'+\theta q)}}  \\
  &\qquad\times\ave{w^{-p'}}_Q^{t(\theta)\frac{1}{t(\theta)p'(1-\theta)}} \ave{w_1^{q_1}}_Q^{u(\theta)\frac{\theta p'}{p_0'(q_1+\theta p')}}  \\
  &=\ave{w^{q}}_Q^{\frac{1}{q(1-\theta)}} \ave{w_1^{-p_1'}}_Q^{\frac{\theta}{p_1'(1-\theta)}}
    \ave{w^{-p'}}_Q^{\frac{1}{p'(1-\theta)}} \ave{w_1^{q_1}}_Q^{\frac{\theta}{q_1(1-\theta)}}  \\
  &=\bigg(\ave{w^q}_Q^{\frac{1}{q}} \ave{w^{-p'}}_Q^{\frac{1}{p'}}\bigg)^{\frac{1}{1-\theta}}
    \bigg(\ave{w_1^{q_1}}_Q^{\frac{1}{q_1}} \ave{w_1^{-p_1'}}_Q^{\frac{1}{p_1'}}\bigg)^{\frac{\theta}{1-\theta}}  \\
  &\leq [w]_{A_{p,q}}^{\frac{1}{1-\theta}}[w_1]_{A_{p_1,q_1}}^{\frac{\theta}{1-\theta}}.
\end{split}
\end{equation*}
In combination with the lines preceding \eqref{eq:beforeRHI1}, we have shown that
\begin{equation*}
  [w_0]_{A_{p_0,q_0}}\lesssim  [w]_{A_{p,q}}^{\frac{1}{1-\theta}}[w_1]_{A_{p_1,q_1}}^{\frac{\theta}{1-\theta}}<\infty,
\end{equation*}
provided that $\theta>0$ is small enough. This concludes the proof.
\end{proof}

\begin{lemma}\label{lem:main2}
Let $1\leq p_{-}<p_{+}<\infty$, $q_1\in[p_{-},p_{+}]$, $q\in(p_{-},p_{+})$, and 
\begin{equation*}
  w_1\in A_{q_1/p_{-}}(\R^d)\cap RH_{(p_{+}/q_1)'}(\R^d),\qquad w\in A_{q/p_{-}}(\R^d)\cap
  RH_{(p_{+}/q)'}(\R^d).
\end{equation*}
Then there exists $q_0\in(p_{-},p_{+})$, $w_0\in A_{q_0/p_{-}}(\R^d)\cap RH_{(p_{+}/q_0)'}(\R^d)$, and $\theta\in(0,1)$ such that (\ref{eq:convexity}) holds.
\end{lemma}

\begin{proof}
By property \eqref{weights prop. 5} of Proposition \ref{prop:weights} we prove the lemma in its equivalence form:
if $v_1:=w_1^{(p_{+}/q_1)'}\in A_{s_1}(\R^d)$ and $v:=w^{(p_{+}/q)'}\in A_{s}(\R^d)$ then there exists $q_0\in(p_{-},p_{+})$, $v_0:=w_0^{(p_{+}/q_0)'}\in A_{s_0}(\R^d)$, and $\theta\in(0,1)$ such that
\begin{equation*}
  [L^{q_0}(w_0),L^{q_1}(w_1)]_\theta=L^q(w),
\end{equation*}
where
\begin{equation*}
  \frac{1}{q}=\frac{1-\theta}{q_0}+\frac{\theta}{q_1},\qquad
  w^{\frac{1}{q}}=w_0^{\frac{1-\theta}{q_0}}w_1^{\frac{\theta}{q_1}},
\end{equation*}
and
\begin{equation*}
\begin{split}
  &s_1=\bigg(\frac{p_{+}}{q_1}\bigg)'\bigg(\frac{q_1}{p_{-}}-1\bigg)+1,  \\ 
  &s=\bigg(\frac{p_{+}}{q}\bigg)'\bigg(\frac{q}{p_{-}}-1\bigg)+1,  \\ 
  &s_0=\bigg(\frac{p_{+}}{q_0}\bigg)'\bigg(\frac{q_0}{p_{-}}-1\bigg)+1.
\end{split}
\end{equation*}

Note that the choice of $\theta\in(0,1)$ determines both 
\begin{equation*}
  q_0=q_0(\theta)=\frac{1-\theta}{\frac{1}{q}-\frac{\theta}{q_1}},\quad
  w_0=w_0(\theta)=w^{\frac{q_0}{q(1-\theta)}}w_1^{-\frac{q_0\cdot\theta}{q_1(1-\theta)}},
\end{equation*}
so it remains to check that we can choose $\theta\in(0,1)$ so that $q_0\in(p_{-},p_{+})$ and $v_0=w_0^{(p_{+}/q_0)'}\in A_{s_0}(\R^d)$, where $s_0=\big(\frac{p_{+}}{q_0}\big)'\big(\frac{q_0}{p_{-}}-1\big)+1$. Since $q_0(0)=q\in(p_{-},p_{+})$, the first condition is obvious for small enough $\theta>0$ by continuity. 

We need to check that $v_0=w_0^{(p_{+}/q_0)'}\in A_{s_0}(\R^d)$, so we consider a cube $Q$ and write
\begin{equation*}
\begin{split}
  \ave{v_0}_Q\ave{v_0^{-\frac{1}{s_0-1}}}_Q&= \ave{w_0^{(p_{+}/q_0)'}}_Q\ave{w_0^{(p_{+}/q_0)'(-\frac{1}{s_0-1})}}_Q^{s_0-1}  \\
  &= \ave{w^{\frac{q_0(p_{+}/q_0)'}{q(1-\theta)}}w_1^{-\frac{q_0\cdot\theta(p_{+}/q_0)'}{q_1(1-\theta)}}}_Q  \\
  &\qquad\times\ave{w^{-\frac{q_0(p_{+}/q_0)'}{q(1-\theta)(s_0-1)}}w_1^{\frac{q_0\cdot\theta(p_{+}/q_0)'}{q_1(1-\theta)(s_0-1)}}}_Q^{s_0-1}.
\end{split}
\end{equation*}

In the first average, we use H\"older's inequality with exponents $1+\eps^{\pm 1}$, and in the second with exponents $1+\delta^{\pm 1}$ for some small enough $\eps,\delta>0$ to get
\begin{equation}\label{eq:beforeRHI2}
\begin{split}
  &\leq \ave{w^{\frac{q_0(p_{+}/q_0)'(1+\eps)}{q(1-\theta)}}}_Q^\frac{1}{1+\eps} \ave{w_1^{-\frac{q_0\theta(p_{+}/q_0)'(1+\eps)}{q_1\eps(1-\theta)}}}_Q^\frac{\eps}{1+\eps}   \\ 
  &\qquad\times\ave{w^{-\frac{q_0(p_{+}/q_0)'(1+\delta)}{q(1-\theta)(s_0-1)}}}_Q^{\frac{s_0-1}{1+\delta}} \ave{w_1^{\frac{q_0\theta(p_{+}/q_0)'(1+\delta)}{q_1\delta(1-\theta)(s_0-1)}}}_Q^\frac{(s_0-1)\delta}{1+\delta}  \\
  &= \ave{(w^{(p_{+}/q)'})^{\tilde r(\theta)}}_Q^\frac{1}{1+\eps} \ave{(w_1^{(p_{+}/q_1)'(-\frac{1}{s_1-1})})^{\tilde s(\theta)}}_Q^\frac{\eps}{1+\eps}  \\
  &\qquad\times\ave{(w^{(p_{+}/q)'(-\frac{1}{s-1})})^{\tilde t(\theta)}}_Q^{\frac{s_0-1}{1+\delta}} \ave{(w_1^{(p_{+}/q_1)'})^{\tilde u(\theta)}}_Q^\frac{(s_0-1)\delta}{1+\delta}  \\
  &=\ave{v^{\tilde r(\theta)}}_Q^\frac{1}{1+\eps} \ave{(v_1^{-\frac{1}{s_1-1}})^{\tilde s(\theta)}}_Q^\frac{\eps}{1+\eps}  \\
  &\qquad\times\ave{(v^{-\frac{1}{s-1}})^{\tilde t(\theta)}}_Q^{\frac{s_0-1}{1+\delta}} \ave{v_1^{\tilde u(\theta)}}_Q^\frac{(s_0-1)\delta}{1+\delta},
\end{split}
\end{equation}
where
\begin{equation*}
  \tilde r(\theta):=\frac{q_0(\theta)(p_{+}-q)(1+\eps)}{q(1-\theta)(p_{+}-q_0(\theta))},\qquad
  \tilde s(\theta):=\frac{\theta q_0(\theta)(p_{+}-q_1)(s_1-1)(1+\eps)}{q_1\eps(1-\theta)(p_{+}-q_0(\theta))}
\end{equation*}
and
\begin{equation*}
  \tilde t(\theta):=\frac{q_0(\theta)(p_{+}-q)(s-1)(1+\delta)}{q(1-\theta)(s_0(\theta)-1)(p_{+}-q_0(\theta))},
\end{equation*}
\begin{equation*}
   \tilde u(\theta):=\frac{\theta q_0(\theta)(p_{+}-q_1)(1+\delta)}{q_1\delta(1-\theta)(s_0(\theta)-1)(p_{+}-q_0(\theta))}.
\end{equation*}

Now, we choose $\eps=\frac{\theta q(p_{+}-q_1)(s_1-1)}{q_1(p_{+}-q)}$ and $\delta=\frac{\theta q(p_{+}-q_1)}{q_1(p_{+}-q)(s-1)}$ in such a way that
\begin{equation*}
  \tilde r(\theta)=\tilde s(\theta)=\frac{q_0(\theta)(q_1(p_{+}-q)+\theta q(p_{+}-q_1)(s_1-1))}{qq_1(1-\theta)(p_{+}-q_0(\theta))},
\end{equation*}
and
\begin{equation*}
  \tilde t(\theta)=\tilde u(\theta)=\frac{q_0(\theta)(q_1(p_{+}-q)(s-1)+\theta q(p_{+}-q_1))}{qq_1(1-\theta)(s_0(\theta)-1)(p_{+}-q_0(\theta))}.
\end{equation*}

The strategy to proceed is the same as in the proof of Lemma \ref{lem:main1}. In particular, we use the reverse H\"older inequality (\ref{eq:RHI}) for $A_v(\R^d)$ weights.

Recalling that $q_0(0)=q$, we see that $\tilde r(0)=\tilde t(0)=1$. By continuity, given any $\eta>0$, we find that
\begin{equation*}
\max(\tilde r(\theta),\tilde t(\theta))\leq 1+\eta\,\,\,\text{for all small enough}\,\,\,\theta>0.
\end{equation*}

By property \eqref{weights prop. 1} of Proposition \ref{prop:weights} each of the four functions $v\in A_s(\R^d)$, $v^{-\frac{1}{s-1}}\in A_{s'}(\R^d)$, $v_1\in A_{s_1}(\R^d)$ and $v_1^{-\frac{1}{s_1-1}}\in A_{s'_1}(\R^d)$ satisfies the reverse H\"older inequality (\ref{eq:RHI}) for all $t\leq 1+\eta$ and for some $\eta>0$. Thus, for all small enough $\theta>0$, we have
\begin{equation*}
\begin{split}
  \eqref{eq:beforeRHI2}
  &\lesssim  \ave{v}_Q^\frac{q_0(p_{+}-q)}{q(1-\theta)(p_{+}-q_0)} \ave{v_1^{-\frac{1}{s_1-1}}}_Q^\frac{\theta q_0(p_{+}-q_1)(s_1-1)}{q_1(1-\theta)(p_{+}-q_0)}  \\
  &\qquad\times\ave{v^{-\frac{1}{s-1}}}_Q^{\frac{q_0(p_{+}-q)(s-1)}{q(1-\theta)(p_{+}-q_0)}} \ave{v_1}_Q^\frac{\theta q_0(p_+-q_1)}{q_1(1-\theta)(p_{+}-q_0)}  \\
  &=(\ave{v}_Q \ave{v^{-\frac{1}{s-1}}}_Q^{s-1})^\frac{q_0(p_{+}-q)}{q(1-\theta)(p_{+}-q_0)}  \\
  &\qquad\times(\ave{v_1}_Q \ave{v_1^{-\frac{1}{s_1-1}}}_Q^{s_1-1})^\frac{\theta q_0(p_{+}-q_1)}{q_1(1-\theta)(p_{+}-q_0)}  \\
  &\leq[v]_{A_s}^\frac{q_1(p_{+}-q)}{p_{+}(q_1-\theta q)-qq_1(1-\theta)}[v_1]_{A_{s_1}}^\frac{\theta q(p_{+}-q_1)}{p_{+}(q_1-\theta q)-qq_1(1-\theta)}.
\end{split}
\end{equation*}
In combination with the lines preceding \eqref{eq:beforeRHI2}, we have shown that
\begin{equation*}
  [v_0]_{A_{s_0}}\lesssim [v]_{A_s}^\frac{q_1(p_{+}-q)}{p_{+}(q_1-\theta q)-qq_1(1-\theta)}[v_1]_{A_{s_1}}^\frac{\theta q(p_{+}-q_1)}{p_{+}(q_1-\theta q)-qq_1(1-\theta)}<\infty,
\end{equation*}
provided that $\theta>0$ is small enough. This concludes the proof.
\end{proof}

\begin{remark}
Lemma \ref{lem:main2} remains true if $p_{+}=\infty$. In this case the reverse H\"older condition on $w_0$ is vacuous and the proof is the same as in \cite[Lemma 4.3]{Hyt}.
\end{remark}

We now have the last missing ingredients of the proofs of Theorems \ref{thm:Off-dig.extrp.compact} and \ref{thm:limited range extrp.compact}:

\begin{proof}[Proof of Proposition \ref{prop:LpvInterm}]
We are given $1<p\leq q<\infty$, $1<p_1\leq q_1<\infty$, and weights $v\in A_{p,q}(\R^d)$, $v_1\in A_{p_1,q_1}(\R^d)$. By Lemma \ref{lem:main1}, there are some $1<p_0\leq q_0<\infty$, a weight $v_0\in A_{p_0,q_0}(\R^d)$, and $\theta\in(0,1)$ such that
\begin{equation}\label{eq:prop}
  \frac{1}{p}=\frac{1-\theta}{p_0}+\frac{\theta}{p_1},\qquad\frac{1}{q}=\frac{1-\theta}{q_0}+\frac{\theta}{q_1},\qquad
  w=w_0^{1-\theta}w_1^{\theta}.
\end{equation}
By Theorem \ref{thm:SW}, we then have 
\begin{equation*}
[L^{p_0}({v_0}^{p_0}),L^{p_1}({v_1}^{p_1})]_\theta=L^p(v^p)\,\,\,\text{and}\,\,\,[L^{q_0}({v_0}^{q_0}),L^{q_1}({v_1}^{q_1})]_\theta=L^q(v^q).
\end{equation*}
Moreover, by \eqref{eq:prop} the claim of the proposition follows.
\end{proof}

\begin{proof}[Proof of Proposition \ref{prop:limitedrange}]
We are given $1\leq p_{-}<p_{+}<\infty$, $q_1\in[p_{-},p_{+}]$, $q\in(p_{-},p_{+})$, and weights $v\in A_{q/p_{-}}(\R^d)\cap RH_{(p_{+}/q)'}(\R^d)$, $v_1\in A_{q_1/p_{-}}(\R^d)\cap RH_{(p_{+}/q_1)'}(\R^d)$. By Lemma \ref{lem:main2}, there is some $q_0\in(p_{-},p_{+})$, a weight $v_0\in A_{q_0/p_{-}}(\R^d)\cap RH_{(p_{+}/q_0)'}(\R^d)$, and $\theta\in(0,1)$ such that
\begin{equation*}
  \frac{1}{q}=\frac{1-\theta}{q_0}+\frac{\theta}{q_1},\qquad
  w^{\frac{1}{q}}=w_0^{\frac{1-\theta}{q_0}}w_1^{\frac{\theta}{q_1}}.
\end{equation*}
By Theorem \ref{thm:SW}, we then have $L^q(v)=[L^{q_0}(v_0),L^{q_1}(v_1)]_\theta$, as we claimed.
\end{proof}

\section{Commutators of fractional integral operators}\label{comm. fr. int. op.}

All our applications of Theorem \ref{thm:Off-dig.extrp.compact} deals with commutators of the form
\begin{equation*}
  [b,T]:f\mapsto bT(f)-T(bf),
\end{equation*}
where the pointwise multiplier $b$ belongs to the space
\begin{equation*}
  \BMO(\R^d):=\Big\{f:\R^d\to\C\ \Big|\ \Norm{f}{\BMO}:=\sup_Q\ave{\abs{f-\ave{f}_Q}}_Q<\infty\Big\}
\end{equation*}
 of functions of bounded mean oscillation, or its subspace
\begin{equation*}
  \CMO(\R^d):=\overline{C_c^\infty(\R^d)}^{\BMO(\R^d)},
\end{equation*}
where the closure is in the $\BMO$ norm. In our first application, we will apply Theorem \ref{thm:Off-dig.extrp.compact} to the commutator $[b,I_\alpha]$, where given $0<\alpha<d$ the fractional integral operator or Riesz potential $I_\alpha$ is defined by
\begin{equation*}
  I_\alpha f(x)=\int_{\R^d}\frac{f(y)}{|x-y|^{d-\alpha}}dy.
\end{equation*}
The weighted norm inequalities for $I_\alpha$ were obtained by Muckenhoupt--Wheeden \cite{MW:74} and the sharp behavior in terms of the weight constants by Lacey--Moen--P\'erez--Torres \cite{Lacey2010}. The commutators of fractional integral operators and $\BMO$ functions were first studied by Chanillo \cite{SC1982}. In \cite{ST:91}, Segovia--Torrea obtained the following weighted commutator result (see Cruz-Uribe and Moen \cite{CM2012} for a sharp quantitative version):

\begin{theorem}[\cite{ST:91}]\label{thm:linear-fractional}
Fix $0<\alpha<d$, $1<p<d/\alpha$, and  $1/p-1/q=\alpha/d$. Suppose also that $b\in\BMO(\R^d)$. Then $[b,I_\alpha]:L^p(w^p)\to L^q(w^q)$ is a bounded linear operator for all $w\in A_{p,q}(\R^d)$.
\end{theorem}

For the application of our Theorem \ref{thm:Off-dig.extrp.compact} we need the result of Wang \cite{W:58} about the compactness of the commutator $[b,I_\alpha]$:

\begin{theorem}[\cite{W:58}]\label{thm:compact-fractional}
If $b\in\CMO(\R^d)$, then $[b,I_\alpha]:L^p(\R^d)\to L^q(\R^d)$ is a compact operator, where $0<\alpha<d$, $1<p<d/\alpha$, and  $1/p-1/q=\alpha/d$.
\end{theorem}

A combination of Theorems \ref{thm:Off-dig.extrp.compact}, \ref{thm:linear-fractional} and \ref{thm:compact-fractional} immediately gives a quick proof of the following recent result of Wu--Yang \cite{WY:2018}:

\begin{corollary}[\cite{WY:2018}, Theorem 1.3]
Let $\alpha\in(0,d),p,q\in(1,\infty)$ with $\frac{1}{p}=\frac{1}{q}+\frac{\alpha}{d},w\in A_{p,q}(\R^d)$ and $b\in\CMO(\R^d)$. Then the commutator $[b,I_\alpha]$ is compact from $L^p(w^p)$ to $L^q(w^q)$.
\end{corollary}

\begin{proof}
Let us fix some $p_1,q_1\in(1,\infty)$ for which we verify the assumptions of Theorem \ref{thm:Off-dig.extrp.compact} for $[b,I_\alpha]$ in place of $T$: By Theorem \ref{thm:linear-fractional}, $[b,I_\alpha]$ is a bounded operator from $L^{p_1}(\tilde w^{p_1})$ to $L^{q_1}(\tilde w^{q_1})$ for all $1<p_1\leq q_1<\infty$ such that $\frac{1}{p}-\frac{1}{q}=\frac{1}{p_1}-\frac{1}{q_1}=\frac{\alpha}{d}$ and all $\tilde w\in A_{p_1,q_1}(\R^d)$. By Theorem \ref{thm:compact-fractional}, $[b,I_\alpha]$ is a compact operator from $L^{p_1}(\R^d)=L^{p_1}(w_1^{p_1})$ to $L^{q_1}(\R^d)=L^{q_1}(w_1^{q_1})$ with $w_1\equiv 1\in A_{p_1,q_1}(\R^d)$. Thus the assumptions, and hence the conclusion, of Theorem \ref{thm:Off-dig.extrp.compact} hold for the operator $[b,I_\alpha]$ in place of $T$, and this is what we wanted.
\end{proof}

The original proof in \cite{WY:2018} relied on verifying the weighted Fr\'echet--Kolmog\-orov criterion \cite{ClopCruz}, providing a sufficient condition for compactness in $L^p(w)$. This is avoided by the aforementioned argument.

Consider now, for $\alpha\in(0,d)$, the so-called {\em $\rho$-type fractional integral operator} defined by
\begin{equation*}
  T_{K_{\alpha}}f(x)=\int_{\R^d}K_{\alpha}(x,y)f(y)dy,\qquad x\notin\supp f,
\end{equation*}
with kernel $K_{\alpha}$ satisfying the size condition
\begin{equation*}
  |K_\alpha(x,y)|\lesssim\frac{1}{|x-y|^{d-\alpha}},
\end{equation*}
and the smooth condition
\begin{equation*}
  |K_{\alpha}(x,y)-K_{\alpha}(z,y)|+|K_{\alpha}(y,x)-K_{\alpha}(y,z)|\leq\rho\bigg(\frac{|x-z|}{|x-y|}\bigg)\frac{1}{|x-y|^{d-\alpha}},
\end{equation*}
for all $x,z,y\in\R^d$ such that $|x-y|>2|x-z|$, where $\rho:[0,1]\to[0,\infty)$ is a modulus of continuity, that is, $\rho$ is a continuous, increasing, subadditive function with $\rho(0)=0$ and satisfies the following Dini condition:
\begin{equation*}
  \int_{0}^{1}\rho(t)\frac{dt}{t}<\infty.
\end{equation*}

By observing that $\abs{T_{K_\alpha}(f)}\lesssim I_{\alpha}(|f|)$ and applying the result of Muckenho\-upt--Wheeden \cite{MW:74} to the operator $I_{\alpha}(|f|)$ we have that $T_{K_\alpha}$ is bounded from $L^p(w^p)$ to $L^q(w^q)$ for all $1<p\leq q<\infty$ such that $\frac{1}{p}-\frac{1}{q}=\frac{\alpha}{d}$ and all weights $w\in A_{p,q}(\R^d)$. We extend this result to the commutator $[b,T_{K_\alpha}]$ by recalling the following result of B\'enyi--Martell--Moen--Stachura--Torres \cite{BMMST} (this is a generalized version of the classical theorem of Coifman--Rochberg--Weiss \cite{CRW1976}):

\begin{theorem}[\cite{BMMST}, Theorem 3.22]\label{thm:commutators lim. range}
Let $T$ be a linear operator. Fix $1\leq p, q<\infty$. Suppose also that $T:L^p(w^p)\to L^q(w^q)$ is bounded for all $w\in A_{p,q}(\R^d)$ and $b\in\BMO(\R^d)$. Then $[b,T]$ is a bounded operator from $L^p(w^p)$ to $L^q(w^q)$.
\end{theorem}

By applying Theorem \ref{thm:commutators lim. range} to the operator $T_{K_\alpha}$ in place of $T$, the following weighted boundedness result for the commutator $[b,T_{K_\alpha}]$ is automatically valid:

\begin{corollary}\label{coro:rho type fractional}
Fix $0<\alpha<d$, $1<p<d/\alpha$ and $1/p-1/q=\alpha/d$. Suppose also that $b\in\BMO(\R^d)$. Then $[b,T_{K_\alpha}]:L^p(w^p)\to L^q(w^q)$ is a bounded linear operator for all $w\in A_{p,q}(\R^d)$.
\end{corollary}

The compactness result about the commutator $[b,T_{K_\alpha}]$ is due to Guo--Wu--Yang \cite{GWY}:

\begin{theorem}[\cite{GWY}, Theorem 1.5]\label{thm:cmpt. of rho type fractional}
Let $w\in A_{p,q}(\R^d)$, $1<p,q<\infty$, $0<\alpha<d$, $1/q=1/p-\alpha/d$. If $b\in\CMO(\R^d)$, then $[b,T_{K_\alpha}]$ is a compact operator from $L^p(w^p)$ to $L^q(w^q)$.
\end{theorem}

The original proof of Theorem \ref{thm:cmpt. of rho type fractional} again follows by applying the weighted Fr\'echet--Kolmogorov criterion obtained in \cite{ClopCruz} and restated in \cite[Lemma 5.4]{GWY}. However, by only applying and verifying the unweighted Fr\'echet--Kolmogorov criterion the proof of Theorem \ref{thm:cmpt. of rho type fractional} can be simplified as follows:

\begin{proof}
Let us fix some $p_1,q_1\in(1,\infty)$ for which we verify the assumptions of Theorem \ref{thm:Off-dig.extrp.compact} for $[b,T_{K_\alpha}]$ in place of $T$: By Corollary \ref{coro:rho type fractional}, $[b,T_{K_\alpha}]$ is a bounded operator from $L^{p_1}(\tilde w^{p_1})$ to $L^{q_1}(\tilde w^{q_1})$ for all $1<p_1\leq q_1<\infty$ such that $\frac{1}{p}-\frac{1}{q}=\frac{1}{p_1}-\frac{1}{q_1}=\frac{\alpha}{d}$ and all $\tilde w\in A_{p_1,q_1}(\R^d)$.
By the unweighted version of \cite[Theorem 1.5]{GWY} (which depends on the classical, unweighted version of the Fr\'echet--Kolmogorov criterion), $[b,T_{K_\alpha}]$ is a compact operator from $L^{p_1}(\R^d)=L^{p_1}({w_1}^{p_1})$ to $L^{q_1}(\R^d)=L^{q_1}({w_1}^{q_1})$ with $w_1\equiv 1\in A_{p_1,q_1}(\R^d)$. Thus the assumptions, and hence the conclusion of Theorem \ref{thm:Off-dig.extrp.compact} hold for the operator $[b,T_{K_\alpha}]$ in place of $T$, and this is what we wanted.
\end{proof}

\section{Commutators of Bochner--Riesz multipliers}\label{comm. br. mult.}
In this section we will apply Theorem \ref{thm:limited range extrp.compact} to the commutators of Bochner--Riesz multipliers in dimensions $d\ge 2$. The latter is a Fourier multiplier $B^\kappa$ with the symbol $(1-|\xi|^2)_{+}^{\kappa}$, where $\kappa>0$ and $t_{+}=\max(t,0)$. That is, the Bochner--Riesz operator is defined, on the class $\testi(\R^d)$ of Schwartz functions, by
\begin{equation*}
  \widehat {B^\kappa f} (\xi)=(1-|\xi|^2)_{+}^{\kappa}\widehat f(\xi),
\end{equation*}
where $\widehat f$ denotes the Fourier transform of $f$.

The following Bochner--Riesz conjecture is well-known:

\begin{conjecture}[Bochner--Riesz Conjecture]\label{j:BR} 
For $0<\kappa<\frac{d-1}{2}$, we have $B^\kappa: L ^{p}(\R ^{d})\mapsto L ^{p}(\R^{d})$ if 
\begin{equation*}
  p\in\bigg(\frac{2d}{d+1+2\kappa},\frac{2d}{d-1-2\kappa}\bigg). 
\end{equation*}
\end{conjecture}

This conjecture holds in two dimensions, as was proved by Carleson--Sj\"olin \cite{CS} (also see C\'ordoba \cite{Cordoba1979}). In the case $d\ge 3$, the best results are currently due to Bourgain--Guth \cite{BG}, but also see Lee \cite{Lee}.

In \cite{LMR2019}, an equivalent form of the Bochner--Riesz Conjecture \ref{j:BR} is stated as follows:
\begin{conjecture}\label{conj.}
Let $\mathbf 1_{[-1/4,1/4]}\leq\chi\leq\mathbf 1_{[-1/2,1/2]}$ be a Schwartz function and denote by $S_{\tau}$ the Fourier multiplier with symbol  $\chi((|\xi|-1)/\tau)$. If $\frac{2d}{d+1}<p<\frac{2d}{d-1}$, then 
\begin{equation}\label{eq:conjecture}
  \|S_{\tau}\|_{L^p(\R^d)\mapsto L^p(\R^d)}\leq C_{\epsilon}\tau^{-\epsilon},
\end{equation}
where $0<\tau<1$ and $C_{\epsilon}$ is a constant that depends on $0<\epsilon<1$.
\end{conjecture}

The connection between the Bochner--Riesz and the $S_{\tau}$ Fourier multipliers is well-known and it can be found in \cite{Carbery, Cordoba1977, Cordoba1979} and \cite[Chapter 8.5]{D2001}. We briefly recall it here. For each $0<\kappa<\frac{d-1}{2}$, we have
\begin{equation*}
  B^\kappa=T^0+\sum_{i=1}^{\infty }2^{-i\kappa}\operatorname{Dil}_{1-2^{-i}}S_{2^{-i}}, 
\end{equation*}
where $T^0$ is a Fourier multiplier, with the multiplier being a Schwartz function supported near the origin and the operator $\operatorname{Dil}_s f(x)=f(x/s)$ is a dilation operator. 
Moreover, each $S_{2^{-i}}$ is a Fourier multiplier with symbol $\chi_i (2^{i}\big||\xi|-1\big|)$, where the $\chi_i$ satisfy a uniform class of derivative estimates.

The partial knowledge of the range of exponent which depends on the parameter $1<p_0<2$ such that the estimate \eqref{eq:conjecture} of Conjecture \ref{conj.} holds is used in the following theorem of Lacey--Mena--Reguera \cite{LMR2019}:

\begin{theorem}[\cite{LMR2019}, Theorem 6.1]\label{thm:Bochner-Riesz limited range estimate}
If $d=2$, $0<\kappa<\tilde\kappa<\frac{1}{2}$ and $p\in(\frac{4}{1+6\kappa},\frac{4}{1-2\kappa})$, then $B^{\tilde\kappa}$ is bounded on $L^p(w)$ for all $w\in A_{\frac{p(1+6\kappa)}{4}}(\R^2)\cap RH_{\big(\frac{4}{p(1-2\kappa)}\big)'}(\R^2)$.
\newline Moreover, if $d\ge 3$, $0<\kappa<\tilde\kappa<\frac{d-1}{2}$, $1<p_0<2$ is such that the estimate \eqref{eq:conjecture} of Conjecture \ref{conj.} holds, and 
\begin{equation*}
  p\in\bigg(\frac{p_0(d-1)}{d-1+2\kappa(p_0-1)},\frac{p_0(d-1)}{d-1-2\kappa}\bigg),
\end{equation*}
then $B^{\tilde\kappa}$ is bounded on $L^p(w)$ for all
\begin{equation*}
  w\in A_{\frac{p(d-1+2\kappa(p_0-1))}{p_0(d-1)}}(\R^d)\cap RH_{\big(\frac{p_0(d-1)}{p(d-1-2\kappa)}\big)'}(\R^d).
\end{equation*}
\end{theorem}

Some earlier results in the same direction are contained in \cite{BBL}, \cite{CDL} and \cite{Christ}.

To streamline the presentation of our main result in this section about the compactness of commutators of Bochner--Riesz multipliers, we formulate the following Corollary of Theorem \ref{thm:Bochner-Riesz limited range estimate}:

\begin{corollary}\label{coro:Bochner-Riesz limited range estimate}
If $d=2$, $0<\tilde\kappa<\frac{1}{2}$ and $p\in(\frac{4}{1+6\tilde\kappa},\frac{4}{1-2\tilde\kappa})$, then $B^{\tilde\kappa}$ is bounded on $L^p(w)$ for all $w\in A_{\frac{p(1+6\tilde\kappa)}{4}}(\R^2)\cap RH_{\big(\frac{4}{p(1-2\tilde\kappa)}\big)'}(\R^2)$.
\newline Moreover, if $d\ge 3$, $0<\tilde\kappa<\frac{d-1}{2}$, $1<p_0<2$ is such that the estimate \eqref{eq:conjecture} of Conjecture \ref{conj.} holds, and 
\begin{equation*}
  p\in\bigg(\frac{p_0(d-1)}{d-1+2\tilde\kappa(p_0-1)},\frac{p_0(d-1)}{d-1-2\tilde\kappa}\bigg),
\end{equation*}
then $B^{\tilde\kappa}$ is bounded on $L^p(w)$ for all
\begin{equation*}
  w\in A_{\frac{p(d-1+2\tilde\kappa(p_0-1))}{p_0(d-1)}}(\R^d)\cap RH_{\big(\frac{p_0(d-1)}{p(d-1-2\tilde\kappa)}\big)'}(\R^d).
\end{equation*} 
\end{corollary}

\begin{proof}
Let us fix $\tilde\kappa, p$ and the weight $w$ of our assumptions. For each selection of these fixed values we show that we can choose $\kappa$ sufficiently close to $\tilde\kappa$ (depending on $\tilde\kappa, p$ and the weight $w$) such that the assumptions of Theorem \ref{thm:Bochner-Riesz limited range estimate} are satisfied. By properties \eqref{weights prop. 2} and \eqref{weights prop. 3} of Proposition \ref{prop:weights} if
\begin{equation*}
  w\in A_{\frac{p(d-1+2\tilde\kappa(p_0-1))}{p_0(d-1)}}(\R^d)\cap RH_{\big(\frac{p_0(d-1)}{p(d-1-2\tilde\kappa)}\big)'}(\R^d)
\end{equation*}
then for $\kappa$ sufficiently close to $\tilde\kappa$ we also have that $\frac{p(d-1+2\kappa(p_0-1))}{p_0(d-1)}$ is sufficiently close to $\frac{p(d-1+2\tilde\kappa(p_0-1))}{p_0(d-1)}$ and $\big(\frac{p_0(d-1)}{p(d-1-2\kappa)}\big)'$ is sufficiently close to
\\$\big(\frac{p_0(d-1)}{p(d-1-2\tilde\kappa)}\big)'$ such that 
\begin{equation*}
  w\in A_{\frac{p(d-1+2\kappa(p_0-1))}{p_0(d-1)}}(\R^d)\cap RH_{\big(\frac{p_0(d-1)}{p(d-1-2\kappa)}\big)'}(\R^d).
\end{equation*}

By continuity, since 
\begin{equation*}
  p\in\bigg(\frac{p_0(d-1)}{d-1+2\tilde\kappa(p_0-1)},\frac{p_0(d-1)}{d-1-2\tilde\kappa}\bigg)
\end{equation*}
we also have that 
\begin{equation*}
  p\in\bigg(\frac{p_0(d-1)}{d-1+2\kappa(p_0-1)},\frac{p_0(d-1)}{d-1-2\kappa}\bigg),
\end{equation*}
provided that $\kappa$ is sufficiently close to $\tilde\kappa$.

Hence the assumptions of Theorem \ref{thm:Bochner-Riesz limited range estimate} are satisfied, and thus $B^{\tilde\kappa}$ is bounded on $L^p(w)$ for the arbitrary choice of the quantities $\tilde\kappa, p$ and $w$ in the statement of Corollary \ref{coro:Bochner-Riesz limited range estimate} that we considered. This concludes the proof.
\end{proof}

We extend this result to the commutator $[b,B^\kappa]$ by recalling the following corollary of Theorem \ref{thm:commutators lim. range} obtained in \cite{BMMST} (it follows by applying property \eqref{weights prop. 5} of Proposition \ref{prop:weights}):

\begin{corollary}[\cite{BMMST}, Corollary 5.3]\label{coro:restricted}
Let $1\leq p_{-}<p<p_{+}\leq\infty$, and $T$ be a linear operator bounded on $L^{p}(w)$ for all $w\in A_{\frac{p}{p_{-}}}(\R^d)\cap RH_{\big(\frac{p_{+}}{p}\big)'}(\R^d)$. If $b\in\BMO(\R^d)$, then $[b,T]$ is bounded on $L^p(w)$ for all $w\in A_{\frac{p}{p_{-}}}(\R^d)\cap RH_{\big(\frac{p_{+}}{p}\big)'}(\R^d)$.
\end{corollary}

By applying Corollary \ref{coro:restricted} to the operator $B^\kappa$ in place of $T$, the following weighted boundedness for the commutator $[b,B^\kappa]$ holds:

\begin{corollary}\label{coro:commu. of Bochner-Riesz limited range estimate}
If $d=2$, $0<\kappa<\frac{1}{2}$, and $p\in(\frac{4}{1+6\kappa},\frac{4}{1-2\kappa})$, then for $b\in\BMO(\R^2)$, the commutator $[b,B^\kappa]$ is bounded on $L^p(w)$ for all $w\in A_{\frac{p(1+6\kappa)}{4}}(\R^2)\cap RH_{\big(\frac{4}{p(1-2\kappa)}\big)'}(\R^2)$.
\newline Moreover, if $d\ge 3$, $0<\kappa<\frac{d-1}{2}$, $1<p_0<2$ is such that the estimate \eqref{eq:conjecture} of Conjecture \ref{conj.} holds and 
\begin{equation*}
  p\in\bigg(\frac{p_0(d-1)}{d-1+2\kappa(p_0-1)},\frac{p_0(d-1)}{d-1-2\kappa}\bigg),
\end{equation*}
then for $b\in\BMO(\R^d)$, the commutator $[b,B^\kappa]$ is bounded on $L^p(w)$ for all
\begin{equation*}
  w\in A_{\frac{p(d-1+2\kappa(p_0-1))}{p_0(d-1)}}(\R^d)\cap RH_{\big(\frac{p_0(d-1)}{p(d-1-2\kappa)}\big)'}(\R^d).
\end{equation*}
\end{corollary}

Moreover, an unweighted compactness result for the commutator $[b,B^\kappa]$ is due to Bu--Chen--Hu \cite{BCH}:

\begin{theorem}[\cite{BCH}, Theorems 1.1 and 1.2]\label{thm:unweighted compt. comm. of Bochner-Riesz limited range estimate}
If $d=2$, $0<\kappa<\frac{1}{2}$, and $p\in(\frac{4}{3+2\kappa},\frac{4}{1-2\kappa})$, then for $b\in\CMO(\R^2)$, the commutator $[b,B^\kappa]$ is compact on $L^p(\R^2)$.
\newline Let $d\ge3$, $\frac{d-1}{2d+2}<\kappa<\frac{d-1}{2}$, and $p\in(\frac{2d}{d+1+2\kappa},\frac{2d}{d-1-2\kappa})$. Then for $b\in\CMO(\R^d)$, the commutator $[b,B^\kappa]$ is compact on $L^p(\R^d)$.
\end{theorem}

Combining Theorem \ref{thm:limited range extrp.compact}, Corollary \ref{coro:commu. of Bochner-Riesz limited range estimate} and Theorem \ref{thm:unweighted compt. comm. of Bochner-Riesz limited range estimate} we can give the new weighted compactness result for the Bochner--Riesz commutator $[b,B^\kappa]$:

\begin{corollary}
If $d=2$, $0<\kappa<\frac{1}{2}$, and $p\in(\frac{4}{1+6\kappa},\frac{4}{1-2\kappa})$, then for $b\in\CMO(\R^2)$, the commutator $[b,B^\kappa]$ is compact on $L^p(w)$ for all $w\in A_{\frac{p(1+6\kappa)}{4}}(\R^2)\cap RH_{\big(\frac{4}{p(1-2\kappa)}\big)'}(\R^2)$.
\newline Moreover, if $d\ge 3$, $\frac{d-1}{2d+2}<\kappa<\frac{d-1}{2}$, $1<p_0<2$ is such that the estimate \eqref{eq:conjecture} of Conjecture \ref{conj.} holds,
\begin{equation*}
  p\in\bigg(\frac{p_0(d-1)}{d-1+2\kappa(p_0-1)},\frac{p_0(d-1)}{d-1-2\kappa}\bigg),   
\end{equation*}
and $b\in\CMO(\R^d)$, then the commutator $[b,B^\kappa]$ is compact on $L^p(w)$ for all
\begin{equation*}
  w\in A_{\frac{p(d-1+2\kappa(p_0-1))}{p_0(d-1)}}(\R^d)\cap RH_{\big(\frac{p_0(d-1)}{p(d-1-2\kappa)}\big)'}(\R^d).
\end{equation*}
\end{corollary}

\begin{proof}
Let $d\ge 3$, $\frac{d-1}{2d+2}<\kappa<\frac{d-1}{2}$ and $p_0$ be as in the assumptions. We verify the assumptions of Theorem \ref{thm:limited range extrp.compact} for the fixed exponent
\begin{equation*}
  \frac{p_0(d-1)}{d-1+2\kappa(p_0-1)}<p_1<\frac{p_0(d-1)}{d-1-2\kappa}
\end{equation*}
and the operator $[b,B^\kappa]$ in place of $T$. By Corollary \ref{coro:commu. of Bochner-Riesz limited range estimate}, $[b,B^\kappa]$ is a bounded operator on $L^{p_1}(\tilde w)$ for all
\begin{equation*}
  \tilde w\in A_{\frac{p_1(d-1+2\kappa(p_0-1))}{p_0(d-1)}}(\R^d)\cap RH_{\big(\frac{p_0(d-1)}{p_1(d-1-2\kappa)}\big)'}(\R^d).
\end{equation*}
By Theorem \ref{thm:unweighted compt. comm. of Bochner-Riesz limited range estimate}, $[b,B^\kappa]$ is a compact operator on $L^{p_1}(\R^d)=L^{p_1}(w_1)$ with
\begin{equation*}
  w_1\equiv 1\in A_{\frac{p_1(d-1+2\kappa(p_0-1))}{p_0(d-1)}}(\R^d)\cap RH_{\big(\frac{p_0(d-1)}{p_1(d-1-2\kappa)}\big)'}(\R^d).
\end{equation*}
Thus Theorem \ref{thm:limited range extrp.compact} applies to give the compactness of $[b,B^\kappa]$ on $L^p(w)$ for all
\begin{equation*}
  p\in\bigg(\frac{p_0(d-1)}{d-1+2\kappa(p_0-1)},\frac{p_0(d-1)}{d-1-2\kappa}\bigg)
\end{equation*}
and all 
\begin{equation*}
  w\in A_{\frac{p(d-1+2\kappa(p_0-1))}{p_0(d-1)}}(\R^d)\cap RH_{\big(\frac{p_0(d-1)}{p(d-1-2\kappa)}\big)'}(\R^d).
\end{equation*}
The case $d=2$ follows in a similar way.
\end{proof}

\section{$A_p^\zeta(\varphi)$ weights and commutators of pseudo-differential operators}\label{sec:Ap(varphi) weights}

In this section, we develop and apply yet another variant for extrapolation of compactness for a special class of weights related to commutators of pseudo-differential operators with smooth symbols.

Following Wu--Wang \cite{WW:2018}, we consider the following:

\begin{definition}
A function $\varphi:[0,\infty)\to[1,\infty)$ is called {\em admissible} if it is non-decreasing and satisfies the following:
\begin{equation*}
  \varphi(\zeta t)\lesssim\zeta^{\omega}\varphi(t),
\end{equation*}
for all $\zeta\ge1$, $t\ge0$ and some $\omega>0$.
\end{definition}

\begin{definition}\label{def:Ap(varphi) weights}
Let $\varphi$ be an admissible function and $p\in(1,\infty)$, $\zeta>0$. A weight $0<w\in L^1_{\loc}(\R^d)$ is called an $A_p^\zeta(\varphi)$ weight (or $w\in A_p^\zeta(\varphi)$) if
\begin{equation*}
  [w]_{A_p^\zeta(\varphi)}:=\sup_Q\frac{\ave{w}_Q\ave{w^{-\frac{1}{p-1}}}_Q^{p-1}}{\varphi(|Q|)^{\zeta p}}<\infty,
\end{equation*}
where the supremum is taken over all cubes $Q\subset\R^d$.
\end{definition}

\begin{remark}
In \cite{Tang}, Tang introduced the weight class $A_p(\varphi)$ which coincides with $A_p^1(\varphi)$. We remark that $A_p^\zeta(\varphi)=A_p(\varphi^\zeta)$. In general, it holds that $A_p(\R^d)\subset A_p^\zeta(\varphi)$ for all $1<p<\infty$. On the other hand, when $\varphi$ is a constant function, $A_p^\zeta(\varphi)=A_p(\R^d)$ for any $\zeta>0$. A main example of admissible function that we consider below is $\varphi(t)=1+t$.
\end{remark}

\subsection{Extrapolation with $A_p^\zeta(\varphi)$ weights}

In \cite[Theorem 2.1]{GZ}, Guo--Zhou proved the compactness of commutators of pseudo-differential operators with smooth symbols on weighted Lebesgue spaces where the weight functions belong to the weight class $A_p^\zeta(\varphi)$. Motivated by their work we show the following extrapolation of compactness:

\begin{theorem}\label{thm:Ap(varphi) compact}
Let $\varphi$ be an admissible function, $1<p<\infty$, and $T$ be a linear operator simultaneously defined and bounded on $L^{p}(w)$ for {\bf all} $1<p<\infty$, {\bf all} $w\in A_p^{\zeta}(\varphi)$ and {\bf all} $\zeta>0$. Suppose in addition that $T$ is compact on $L^{p_1}(w_1)$ for {\bf some} $1<p_1<\infty$, {\bf some} $w_1\in A_{p_1}^{\zeta_1}(\varphi)$ and {\bf some} $\zeta_1>0$. Then $T$ is compact on $L^p(w)$ for all $p\in (1,\infty)$, all $w\in A_p^{\zeta}(\varphi)$ and all $\zeta>0$.
\end{theorem}

We proceed by collecting the results from which the proof of Theorem \ref{thm:Ap(varphi) compact} follows. We will use Theorem \ref{thm:CwKa} in the special setting:

\begin{proposition}\label{prop:Ap(varphi)}
Let $\varphi$ be an admissible function and suppose that $q,q_1\in (1,\infty),\zeta,\zeta_1>0$, $v\in A_q^\zeta(\varphi)$, $v\in A_{q_1}^{\zeta_1}(\varphi)$. Then
\begin{equation*}
  [L^{q_0}(v_0),L^{q_1}(v_1)]_{\gamma}=L^q(v)
\end{equation*}
for some $q_0\in(1,\infty), \zeta_0>0$, $v_0\in A_{q_0}^{\zeta_0}(\varphi)$, and $\gamma\in(0,1)$.
\end{proposition}

The only component of the proof of Theorem \ref{thm:Ap(varphi) compact} that requires actual computations is the verification of this proposition. For this purpose we will need Theorem \ref{thm:SW} which we connect it with the $A_p^\zeta(\varphi)$ weights as follows:

\begin{lemma}\label{lem:main3}
Let $\varphi$ be an admissible function and $p_1,p\in(1,\infty),\zeta,\zeta_1>0$, $w_1\in A_{p_1}^{\zeta_1}(\varphi)$, $w\in A_p^{\zeta}(\varphi)$. Then there exists $p_0\in(1,\infty),\zeta_0>0$, $w_0\in A_{p_0}^{\zeta_0}(\varphi)$, and $\theta\in(0,1)$ such that the conclusion of Theorem \ref{thm:SW} holds, i.e.,
\begin{equation*}
  [L^{p_0}(w_0),L^{p_1}(w_1)]_\theta=L^p(w),
\end{equation*}
where
\begin{equation*}
  \frac{1}{p}=\frac{1-\theta}{p_0}+\frac{\theta}{p_1},\qquad
  w^{\frac{1}{p}}=w_0^{\frac{1-\theta}{p_0}}w_1^{\frac{\theta}{p_1}}.
\end{equation*}
\end{lemma}

\begin{proof}
Note that the choice of $\theta\in(0,1)$ determines both
\begin{equation*}
  p_0=p_0(\theta)=\frac{1-\theta}{\frac{1}{p}-\frac{\theta}{p_1}},\qquad
  w_0=w_0(\theta)=w^{\frac{p_0}{p(1-\theta)}}w_1^{-\frac{p_0\cdot\theta}{p_1(1-\theta)}},
\end{equation*}
so it remains to check that we can choose $\theta\in(0,1)$ so that $p_0\in(1,\infty)$ and $w_0\in A_{p_0}^{\zeta_0}(\varphi)$ for some $\zeta_0>0$. Since $p_0(0)=p\in(1,\infty)$, the first condition is obvious for small enough $\theta>0$ by continuity.

To check that $w_0\in A_{p_0}^{\zeta_0}(\varphi)$ for some $\zeta_0>0$, we consider a cube $Q$ and write
\begin{equation*}
\begin{split}
  &\ave{w_0}_Q \ave{w_0^{-\frac{1}{p_0-1}}}_Q^{p_0-1}
  =\ave{w^{\frac{p_0}{p(1-\theta)}}w_1^{-\frac{p_0\cdot\theta}{p_1(1-\theta)}}}_Q
  \ave{w^{-\frac{p_0'}{p(1-\theta)}}w_1^{\frac{p_0'\cdot\theta}{p_1(1-\theta)}}}_Q^{p_0-1}  \\
  &=\ave{w^{\frac{p_0}{p(1-\theta)}}(w_1^{-\frac{1}{p_1-1}})^{\frac{p_0\cdot\theta}{p_1'(1-\theta)}}}_Q
  \ave{(w^{-\frac{1}{p-1}})^{\frac{p_0'}{p'(1-\theta)}}w_1^{\frac{p_0'\cdot\theta}{p_1(1-\theta)}}}_Q^{p_0-1},
\end{split}
\end{equation*}
where $q':=q/(q-1)$ denotes the conjugate exponent of $q\in\{p,p_0,p_1\}$.

In the first average, we use H\"older's inequality with exponents $1+\eps^{\pm 1}$, and in the second with exponents $1+\delta^{\pm 1}$ to get
\begin{equation*}
\begin{split}
  &\leq \ave{w^{\frac{p_0(1+\eps)}{p(1-\theta)}}}_Q^{\frac{1}{1+\eps}}  \ave{(w_1^{-\frac{1}{p_1-1}})^{\frac{p_0\theta(1+\eps)}{p_1'\eps(1-\theta)}}}_Q^{\frac{\eps}{1+\eps}}
  \ave{(w^{-\frac{1}{p-1}})^{\frac{p_0'(1+\delta)}{p'(1-\theta)}} }_Q^{\frac{p_0-1}{1+\delta}}  \\ 
  &\qquad\times\ave{ w_1^{\frac{p_0'\theta(1+\delta)}{p_1\delta(1-\theta)}}}_Q^{\frac{\delta(p_0-1)}{1+\delta}}.
\end{split}
\end{equation*}
If we choose $\eps=\theta p/p_1'$ and $\delta=\theta p'/p_1$, the previous line takes the form
\begin{equation}\label{eq:beforeRHI3}
\begin{split}
  &=\ave{w^{r(\theta)}}_Q^{\frac{p_1'}{p_1'+\theta p}}\ave{(w_1^{-\frac{1}{p_1-1}})^{r(\theta)}}_Q^{\frac{\theta p}{p_1'+\theta p}}
  \ave{(w^{-\frac{1}{p-1}})^{s(\theta)} }_Q^{\frac{p_1(p_0-1)}{p_1+\theta p'}}  \\
  &\qquad\times\ave{ w_1^{s(\theta)}}_Q^{\frac{\theta p'(p_0-1)}{p_1+\theta p'}},
\end{split}
\end{equation}
where
\begin{equation*}
  r(\theta):=\frac{p_0(\theta)(p_1'+\theta p)}{p\cdot p_1'(1-\theta)},\qquad
  s(\theta):=\frac{p_0(\theta)'(p_1+\theta p')}{ p' p_1(1-\theta)}.
\end{equation*}

The strategy to proceed is to use the reverse H\"older inequality for $A_v^{\tilde\zeta}(\varphi)$ weights due to Wu--Wang \cite[Proposition 15]{WW:2018}, which says that for each $W\in A_v^{\tilde\zeta}(\varphi)$ there exists $\eta>0$ such that
\begin{equation}\label{eq:RHI2}
  \ave{W^t}_Q^{1/t}\lesssim \ave{W}_{Q}\varphi(|Q|)^{\eta}
\end{equation}
for all $t\leq 1+\tilde\eta$ and for some $\tilde\eta>0$.

Recalling that $p_0(0)=p$, we see that $r(0)=1=s(0)$. By continuity, given any $\tilde\eta>0$, we find that 
\begin{equation}\label{eq:new weights}
  \max(r(\theta),s(\theta))\leq 1+\tilde\eta\,\,\,\text{for all small enough}\,\,\,\theta>0.
\end{equation}

Next we will apply another property of $A_v^{\tilde\zeta}(\varphi)$ weights as stated in Wu--Wang \cite[Proposition 15]{WW:2018}, namely:
\newline If $1<v<\infty$, we have 
\begin{equation}\label{Ap(zeta)(phi) prop.}
  W\in A_v^{\tilde\zeta}(\varphi)\Longleftrightarrow W^{1-v'}\in A_{v'}^{\tilde\zeta}(\varphi),\qquad\text{where}\qquad\frac{1}{v}+\frac{1}{v'}=1.
\end{equation}

By \eqref{Ap(zeta)(phi) prop.} we have that $w\in A_p^\zeta(\varphi)$, $w_1^{-\frac{1}{p_1-1}}\in A_{p_1'}^{\zeta_1}(\varphi)$, $w^{-\frac{1}{p-1}}\in A_{p'}^\zeta(\varphi)$, and $w_1\in A_{p_1}^{\zeta_1}(\varphi)$. Hence by \eqref{eq:new weights} each of these four functions satisfies the reverse H\"older inequality \eqref{eq:RHI2} for all $t\leq 1+\tilde\eta$ and for some $\tilde\eta>0$. Thus, for all small enough $\theta>0$, we have
\begin{equation*}
\begin{split}
  \eqref{eq:beforeRHI3}
  &\lesssim \ave{w}_Q^{r(\theta)\frac{p_1'}{p_1'+\theta p}} \ave{w_1^{-\frac{1}{p_1-1}}}_Q^{r(\theta)\frac{\theta p}{p_1'+\theta p}}
  \ave{w^{-\frac{1}{p-1}} }_Q^{s(\theta)\frac{p_1(p_0-1)}{p_1+\theta p'}}  \\
  &\qquad\times\ave{ w_1 }_Q^{s(\theta)\frac{\theta p'(p_0-1)}{p_1+\theta p'}}\varphi(|Q|)^{\eta r(\theta)+\eta s(\theta)(p_0-1)}  \\
  &=\ave{w}_Q^{\frac{p_0(\theta) }{p(1-\theta)}} \ave{w_1^{-\frac{1}{p_1-1}}}_Q^{\frac{\theta p_0(\theta)}{p_1'(1-\theta)}}\ave{w^{-\frac{1}{p-1}} }_Q^{\frac{p_0(\theta)}{p'(1-\theta)}}  \\
  &\qquad\times\ave{ w_1 }_Q^{\frac{\theta p_0(\theta)}{p_1(1-\theta)}}\varphi(|Q|)^{\eta r(\theta)+\eta s(\theta)(p_0(\theta)-1)} \\
  &=(\ave{w}_Q\ave{w^{-\frac{1}{p-1}} }_Q^{p-1})^{\frac{p_0(\theta) }{p(1-\theta)}}
    (\ave{ w_1 }_Q\ave{w_1^{-\frac{1}{p_1-1}}}_Q^{p_1-1})^{\frac{\theta p_0(\theta)}{p_1(1-\theta)}}  \\
  &\qquad\times\varphi(|Q|)^{\eta r(\theta)+\eta s(\theta)(p_0(\theta)-1)}  \\
  &\leq [w]_{A_{p}^{\zeta}(\varphi)}^{\frac{p_1}{p_1-\theta p}}[w_1]_{A_{p_1}^{\zeta_1}(\varphi)}^{\frac{\theta p}{p_1-\theta p}}\varphi(|Q|)^{\zeta_0 p_0(\theta)},
\end{split}
\end{equation*}
where $\zeta_0=\eta\frac{r(\theta)+s(\theta)(p_0(\theta)-1)}{p_0(\theta)}+\frac{\zeta}{1-\theta}+\frac{\zeta_1 \theta}{1-\theta}>0$. In combination with the lines preceding \eqref{eq:beforeRHI3}, we have shown that
\begin{equation*}
   [w_0]_{A_{p_0}^{\zeta_0}(\varphi)}\lesssim [w]_{A_{p}^{\zeta}(\varphi)}^{\frac{p_1}{p_1-\theta p}} [w_1]_{A_{p_1}^{\zeta_1}(\varphi)}^{\frac{\theta p}{p_1-\theta p}}<\infty,
\end{equation*}
provided that $\theta>0$ is small enough. This concludes the proof. 
\end{proof}

We now have the last missing ingredient of the proof of Theorem \ref{thm:Ap(varphi) compact}:

\begin{proof}[Proof of Proposition \ref{prop:Ap(varphi)}]
We are given $q,q_1\in (1,\infty),\zeta,\zeta_1>0$, and weights $v\in A_q^\zeta(\varphi)$, $v_1\in A_{q_1}^{\zeta_1}(\varphi)$. By Lemma \ref{lem:main3}, there is some $q_0\in (1,\infty),\zeta_0>0$, a weight $v_0\in A_{q_0}^{\zeta_0}(\varphi)$, and $\theta\in(0,1)$ such that
\begin{equation*}
  \frac{1}{q}=\frac{1-\theta}{q_0}+\frac{\theta}{q_1},\qquad
  w^{\frac{1}{q}}=w_0^{\frac{1-\theta}{q_0}}w_1^{\frac{\theta}{q_1}}.
\end{equation*}
By Theorem \ref{thm:SW}, we then have $L^q(v)=[L^{q_0}(v_0),L^{q_1}(v_1)]_\theta$, as we claimed.
\end{proof}

By combining Theorem \ref{thm:CwKa}, Lemma \ref{lem:lemOk} and Proposition \ref{prop:Ap(varphi)} we can prove Theorem \ref{thm:Ap(varphi) compact} as follows:

\begin{proof}[Proof of Theorem \ref{thm:Ap(varphi) compact}]
Recall that the assumptions of Theorem \ref{thm:Ap(varphi) compact} are in force. In particular, $T$ is a bounded linear operator on $L^p(w)$ for all $p\in(1,\infty)$, all $w\in A_p^\zeta(\varphi)$ and all $\zeta>0$. In addition, it is assumed that $T$ is a compact operator on $L^{p_1}(w_1)$ for some $p_1\in(1,\infty)$, some $w_1\in A_{p_1}^{\zeta_1}(\varphi)$ and some $\zeta_1>0$. We need to prove that $T$ is actually compact on $L^p(w)$ for all $p\in(1,\infty)$, all $w\in A_p^\zeta(\varphi)$ and all $\zeta>0$. By Proposition \ref{prop:Ap(varphi)}, we have
\begin{equation*}
  L^p(w)=[L^{p_0}(w_0),L^{p_1}(w_1)]_\theta
\end{equation*}
for some $p_0\in(1,\infty)$, some $\zeta_0>0$, some $w_0\in A_{p_0}^{\zeta_0}(\varphi)$, and some $\theta\in(0,1)$. Writing $X_j=Y_j=L^{p_j}(w_j)$, we know that $T:X_0+X_1\to Y_0+Y_1$, that $T:X_j\to Y_j$ is bounded, and that $T:X_1\to Y_1$ is compact (since the last two assertions were assumed). By Lemma \ref{lem:lemOk}, the last condition \eqref{it:lattice} of Theorem \ref{thm:CwKa} is also satisfied by these spaces $X_j=Y_j=L^{p_j}(w_j)$. By Theorem \ref{thm:CwKa}, it follows that $T$ is also compact on $[X_0,X_1]_\theta=[Y_0,Y_1]_\theta=L^p(w)$.
\end{proof}

We provide an application of Theorem \ref{thm:Ap(varphi) compact} that concerns pseudo-differential operators with smooth symbols.

\subsection{Commutators of pseudo-differential operators with smooth symbols}

Following \cite{Taylor}, we say that a symbol $\sigma$ belongs to $S_{1,\lambda}^m$ if $\sigma(x,\xi)$ is a smooth function of $(x,\xi)\in\R^d\times\R^d$ and satisfies the following estimate:
\begin{equation*}
  \abs{\partial_x^\mu\partial_\xi^\nu\sigma(x,\xi)}\lesssim(1+\abs{\xi})^{m-\abs{\nu}+\lambda|\mu|},
\end{equation*}
for all $\mu,\nu\in \N^d$, where $m\in\R$.

Let $\sigma(x,\xi)\in S_{1,\lambda}^m$. The pseudo-differential operator $T$ is defined by
\begin{equation*}
  Tf(x)=\int_{\R^d}\sigma(x,\xi)e^{2\pi ix\cdot\xi}\widehat f(\xi)d\xi,
\end{equation*}
where $f$ is a Schwartz function and $\widehat f$ denotes the Fourier transform of $f$. As usual, $L_{1,\lambda}^m$ will denote the class of pseudo-differential operators with symbols in $S_{1,\lambda}^m$.

Miller \cite{Miller} showed the boundedness of $L_{1,0}^0$ pseudo-differential operators on $L^p(w)$ for $1<p<\infty$ and $w\in A_p(\R^d)$. Tang \cite{Tang} improved the results of Miller by showing the boundedness of $L_{1,0}^0$ pseudo-differential operators and their commutators on $L^p(w)$, where $w\in A_p^\zeta(\varphi)$, $\varphi(t)=1+t$ and $\zeta>0$ (Tang also makes a remark about the case $L_{1,\lambda}^0$ ($0<\lambda<1$); see \cite[after Corollary 1.2]{Tang}).

We will apply Theorem \ref{thm:Ap(varphi) compact} to the commutators of pseudo-differential operators $T\in L_{1,0}^0$. We need the following result of Tang \cite{Tang}:

\begin{theorem}[\cite{Tang}, Theorem 1.2]\label{thm:bdd. of comm. of pseudo-diff. oper. with smooth symbols}
Suppose that $T\in L_{1,0}^0$. Let $b\in\BMO(\R^d)$, $1<p<\infty$. Then $[b,T]$ is bounded on $L^p(w)$ for all $w\in A_p^\zeta(\varphi)$, where $\varphi(t)=1+t$ and $\zeta>0$.
\end{theorem}
 
By \cite[Th\'eor\`eme 19]{CM:audela} these operators are instances of Calder\'on--Zygmund operators, namely:
\begin{equation*}
  Tf(x)=\int_{\R^d}K(x,y)f(y)\ud y,\qquad x\notin\supp f,
\end{equation*}
where $T$ is a linear operator defined on a suitable class of test functions on $\R^d$ and the kernel $K$ satisfies the {\em standard estimates} 
\begin{equation*}
  \abs{K(x,y)}\lesssim\frac{1}{\abs{x-y}^d}
\end{equation*}
and, for some $\delta_0\in(0,1]$,
\begin{equation*}
  \abs{K(x,y)-K(z,y)}+\abs{K(y,x)-K(y,z)}\lesssim \frac{\abs{x-z}^{\delta_0}}{\abs{x-y}^{d+\delta_0}},
\end{equation*}
for all $x,z,y\in\R^d$ such that $\abs{x-y}>\frac{1}{2}\abs{x-z}$. 

The following result about the compactness in $L^p(\R^d)$ for the commutators of Calder\'on--Zygmund operators is due to Uchiyama \cite{Uchi}:

\begin{theorem}[\cite{Uchi}]\label{thm:Uchi}
Let $T$ be a Calder\'on--Zygmund operator that extends to a bounded operator on $L^2(\R^d)$. If $b\in\CMO(\R^d)$, then $[b,T]$ is compact on the unweighted $L^p(\R^d)$ for all $p\in(1,\infty)$.
\end{theorem}

By applying Theorem \ref{thm:Ap(varphi) compact} we can now recover a very recent result of Guo--Zhou \cite{GZ}:

\begin{theorem}[\cite{GZ}, Theorem 2.1] Suppose that $T\in L_{1,0}^0$. Let $b\in CMO(\R^d)$, $1<p<\infty$. Then the commutator $[b,T]$ is a compact operator on $L^p(w)$ for all $w\in A_p^\zeta(\varphi)$, where $\varphi(t)=1+t$ and $\zeta>0$.
\end{theorem}

\begin{proof}
We verify the assumptions of Theorem \ref{thm:Ap(varphi) compact} for $[b,T]$ in place of $T$: By Theorem \ref{thm:bdd. of comm. of pseudo-diff. oper. with smooth symbols} $[b,T]$ is a bounded operator on $L^p(w)$ for all $1<p<\infty$, all $w\in A_p^\zeta(\varphi)$ and all $\zeta>0$. By Theorem \ref{thm:Uchi}, $[b,T]$ is a compact operator on $L^{p_1}(\R^d)=L^{p_1}(w_1)$ for any $1<p_1<\infty$ with $w_1\equiv 1\in A_{p_1}^{\zeta_1}(\varphi)$ and any $\zeta_1>0$. Thus Theorem \ref{thm:Ap(varphi) compact} applies to give the compactness of $[b,T]$ on $L^p(w)$ for all $p\in (1,\infty)$, all $w\in A_p^\zeta(\varphi)$ and all $\zeta>0$.
\end{proof}

As in the case of the commutators of fractional integral operators in Section \ref{comm. fr. int. op.} the original proof in \cite{GZ} relied on verifying the weighted Fr\'echet--Kolmogorov criterion \cite{ClopCruz}, which is avoided by the argument above.


\end{document}